\documentclass[reqno]{amsart}
\usepackage{amsmath}
\usepackage[dvips]{graphicx}
\usepackage{amsfonts}
\usepackage{amssymb}
\usepackage{latexsym}
\graphicspath{{Img/}}

\newtheorem{theorem}{Theorem}

\theoremstyle{plain}

\newtheorem{corollary}{Corollary}

\newtheorem{problem}{Problem}
\newtheorem{proposition}{Proposition}
\newtheorem{remark}{Remark}

\numberwithin{equation}{section}

\begin{document}
\title[An $\epsilon$ characterization of a vertex on surfaces]{An $\epsilon$ characterization of a vertex formed by two non-overlapping geodesic arcs on surfaces with constant Gaussian curvature}
\author{Anastasios N. Zachos}
\address{Greek Ministry of Education, Athens, Greece}
\email{azachos@gmail.com}

\begin{abstract}
We determine a positive real number (weight) which corresponds to the intersection point (vertex) of two non-overlapping geodesic arcs, which depends on the two weights which correspond to two points of these geodesic arcs, respectively, and an infinitesimal number $\epsilon.$ As a limiting case for $\epsilon\to 0,$ the triad of the corresponding weights yields a degenerate weighted Fermat-Torricelli tree which coincides with these two geodesic arcs. By applying this process for a geodesic triangle on a circular cone, we derive an $\epsilon$ characterization of conical points in $\mathbb{R}^{3}.$

\end{abstract}\maketitle

\section{Introduction}

Let $\triangle A_{1}A_{2}A_{3}$ be a geodesic triangle on a surface $S$ with constant Gaussian curvature in $\mathbb{R}^{3}.$
We denote by $A_{0}$ a point on $S,$ by $(a_{ij})_{g}$ the length of the geodesic arc $A_{i}A_{j}$  by $\angle A_{i}A_{j}A_{k}$ the angle formed by the geodesic arcs $A_{i}A_{j}$ and $A_{j}A_{k},$
at the vertex $A_{j},$ by $\vec{U}_{ij}$ is the unit tangent
vector of the geodesic arc $A_{i}A_{j}$ at $A_{i},$ for $i,j,k=0,1,2,3$ and by $B_{i}$ a positive real number (weight), which corresponds to each vertex $A_{i},$ for $i=1,2,3.$

We state the weighted Fermat-Torricelli problem for a $\triangle A_{1}A_{2}A_{3}$ on $S.$

\begin{problem}   Find a point $A_{0}$ (weighted Fermat-Torricelli
point) on $S,$  such that
\begin{equation}\label{minimumcs}
 f(A_{0})=\sum_{i=1}^{3}B_{i}(a_{i0})_{g} \to\ \min.
\end{equation}

\end{problem}

The solution of the weighted Fermat-Torricelli problem is called a branching solution (weighted Fermat-Torricelli tree) and consists of the three geodesic branches $\{A_{1}A_{0},A_{2}A_{0},A_{3}A_{0}\}$ which meet at the weighted
Fermat-Torricelli point $A_{0}.$

The following two propositions give a characterization of the weighted Fermat-Torricelli point $A_{0}$ on a smooth surface, which have been proved in \cite{Cots/Zach:11},\cite{Zachos/Cots:10}).

\begin{proposition}[Floating Case]\cite{Cots/Zach:11},\cite{Zachos/Cots:10})
The following (I), (II), (III) conditions are equivalent:\\

(I) All the following inequalities are satisfied simultaneously:
\begin{equation}\label{cond120n}
\left\| B_{2}\vec{U}_{12}+B_{3}\vec{U}_{13}\right\|> B_{1},
\end{equation}

\begin{equation}\label{cond1202n}
\left\| B_{1}\vec{U}_{21}+B_{3}\vec{U}_{23}\right\|> B_{2},
\end{equation}

\begin{equation}\label{cond1203n}
\left\| B_{1}\vec{U}_{31}+B_{2}\vec{U}_{32}\right\|> B_{3},
\end{equation}
(II) The point $A_{0}$ is an interior point of the triangle
$\triangle A_{1}A_{2}A_{3}$ and does not belong to the geodesic arcs
$A_{i}A_{j}$ for $i,j=1,2,3$

(III) $B_{1}\vec{U}_{01}+B_{2}\vec{U}_{02}+B_{3}\vec{U}_{03}=\vec{0}.$

\end{proposition}

\begin{proposition}[Absorbed Case]\cite{Cots/Zach:11},\cite{Zachos/Cots:10})
The following (I), (II) conditions are equivalent.\\(I) One of the
following inequalities is satisfied:
\begin{equation}\label{cond120}
\left\| B_{2}\vec{U}_{12}+B_{3}\vec{U}_{13}\right\|\le B_{1},
\end{equation}
or
\begin{equation}\label{cond1202}
\left\| B_{1}\vec{U}_{21}+B_{3}\vec{U}_{23}\right\|\le B_{2},
\end{equation}
or
\begin{equation}\label{cond1203}
\left\| B_{1}\vec{U}_{31}+B_{2}\vec{U}_{32}\right\|\le B_{3}.
\end{equation}

(II) The point $A_{0}$ is attained at $A_{1}$ or $A_{2}$ or $A_{3},$
respectively.

\end{proposition}

The inverse weighted Fermat-Torricelli problem on $S$ states that:

\begin{problem}\cite{Cots/Zach:11},\cite{Zachos/Cots:10}
Given a point $A_{0}$ which belongs to the interior of $\triangle A_{1}A_{2}A_{3}$ on $M$,
does there exist a unique set of positive weights $B_{i},$ such
that
\begin{displaymath}
 B_{1}+B_{2}+B_{3}= c =const,
\end{displaymath}
for which $A_{0}$ minimizes
\begin{displaymath}
 f(A_{0})=B_{1}(a_{10})_{g}+B_{2}(a_{20})_{g}+B_{3}(a_{30})_{g}
\end{displaymath}
\end{problem}

A positive answer w.r. to the inverse weighted Fermat-Torricelli problem on a $C^{2}$ complete surface is given by the following proposition (\cite{Cots/Zach:11},\cite{Zachos/Cots:10}):
\begin{proposition}\label{propo5}\cite{Cots/Zach:11},\cite{Zachos/Cots:10}
The weight $B_{i}$ are uniquely determined by the formula:
\begin{equation}\label{inverse111}
B_{i}=\frac{C}{1+\frac{\sin{\angle A_{i}A_{0}A_{j}}}{\sin{\angle A_{j}A_{0}A_{k}}}+\frac{\sin{\angle A_{i}A_{0}A_{k}}}{\sin{\angle A_{j}A_{0}A_{k}}}},
\end{equation}
for $i,j,k=1,2,3$ and $i \neq j\neq k.$
\end{proposition}

If $B_{1},$ $B_{2},$ $B_{3}$ satisfy the inequalities of the floating case, we derive
a weighted Fermat-Torricelli tree $\{A_{1}A_{0}, A_{2}A_{0}, A_{3}A_{0}\}.$

The location of the weighted Fermat-Torricelli (floating) tree for geodesic triangles on the $K$ plane (sphere $S_{K}^2$, hyperboloid $H^2$) is given in \cite{Zachos:13}, \cite{Zachos 14b} and an analytical solution of the weighted Fermat-Torricelli
problem for an equilateral geodesic triangle with equal lengths $\frac{\pi}{2}$ is given in \cite{Zach:15}, for the weighted floating case. Concerning, the solution of the weighted Fermat-Torricelli problem for geodesic triangles on flat surfaces of revolution (Circular cylinder, circular cone) we refer to \cite{Zach:14a}.

If $B_{1},$ $B_{2},$ $B_{3}$ satisfy one of the inequalities of the absorbed case (\ref{cond120}) or (\ref{cond1202}) or (\ref{cond1203}), we obtain
a degenerate weighted Fermat-Torricelli tree $\{A_{2}A_{1},A_{1}A_{3}\},$ $\{A_{1}A_{2},A_{2}A_{3}\}$
and $\{A_{1}A_{3},A_{3}A_{2}\},$ respectively.


For instance, if (\ref{cond120}) is valid, the minimum value of $B_{1}$
is determined by:

\[B_{1}^2=B_{2}^2+B_{3}^2+2B_{2}B_{3}\cos\angle A_{2}A_{1}A_{3}\]
 or

\[B_{1}=f(B_{2},B_{3}).\]

Thus, we consider the following problem:

\begin{problem}
How can we determine the values of $B_{2},$ $B_{3},$ such that $f(B_{2},B_{3})$ gives the minimum value of
$B_{1}$ that corresponds to the vertex $A_{1}$ on a surface with constant Gaussian curvature $S$?
\end{problem}





In this paper, we determine the value of $B_{1}$ by introducing an infinitesimal real number $\epsilon,$ ($\epsilon$  characterization of $A_{1}$)
such that:
$\angle A_{1}A_{2}A_{0}= \|\epsilon\|,$ $\angle A_{2}A_{0}A_{3}= \angle A_{2}A_{1}A_{3}+2\epsilon,$
$\angle A_{0}A_{3}A_{1}= k\|\epsilon\|$ for a rational number $k,$ by applying the solution of the inverse weighted
Fermat-Torricelli problem on $S.$
By setting $A_{1}$ to be the vertex of a (right) circular cone, we give an $\epsilon$ characterization of conical points in $\mathbb{R}^{3}.$


\section{An $\epsilon$ characterization of the vertices of a triangle in $\mathbb{R}^{2}$}

Let $A_{0}$  be an interior point of $\triangle A_{1}A_{2}A_{3}$ in $\mathbb{R}^2.$

We denote by $(a_{ij})_{0}$ the length of the linear segment
$A_{i}A_{j}.$

We set $\angle A_{1}A_{2}A_{0}= \epsilon,$ $\angle A_{2}A_{0}A_{3}= \angle A_{2}A_{1}A_{3}+2\epsilon$
$\angle A_{0}A_{3}A_{1}= k\epsilon.$



\begin{theorem}\label{echarR2}

The weight $B_{i}=B_{i}(\epsilon)$ are uniquely determined by the formula:
\begin{equation}\label{inverseB1e}
B_{1}=\frac{C}{1+\frac{\sin{\angle A_{1}A_{0}A_{3}}}{\sin{\angle A_{2}A_{0}A_{3}}}+\frac{\sin{\angle A_{1}A_{0}A_{2}}}{\sin{\angle A_{2}A_{0}A_{3}}}},
\end{equation}

\begin{equation}\label{inverseB2e}
B_{2}=\frac{C}{1+\frac{\sin{\angle A_{2}A_{0}A_{3}}}{\sin{\angle A_{1}A_{0}A_{3}}}+\frac{\sin{\angle A_{2}A_{0}A_{1}}}{\sin{\angle A_{1}A_{0}A_{3}}}},
\end{equation}

\begin{equation}\label{inverseB3e}
B_{3}=\frac{C}{1+\frac{\sin{\angle A_{3}A_{0}A_{1}}}{\sin{\angle A_{1}A_{0}A_{2}}}+\frac{\sin{\angle A_{3}A_{0}A_{2}}}{\sin{\angle A_{1}A_{0}A_{2}}}},
\end{equation}

where

\begin{equation}\label{alpha203}
\angle A_{2}A_{0}A_{3}= \angle A_{2}A_{1}A_{3}+2\epsilon
\end{equation}

\begin{equation}\label{alpha102}
\angle A_{1}A_{0}A_{2}=\angle A_{1}A_{0}A_{2}(\epsilon)=\operatorname{arccot} (-\frac{(a_{13})_{0}+(a_{12})_{0}\cos(\angle A_{2}A_{1}A_{3}+2\epsilon)}{(a_{12})_{0}\sin(\angle A_{2}A_{1}A_{3}+2\epsilon)}),
\end{equation}

\begin{equation}\label{alpha103}
\angle A_{1}A_{0}A_{3}=2\pi- \angle A_{2}A_{1}A_{3}-2\epsilon-\angle A_{1}A_{0}A_{2}(\epsilon).
\end{equation}

\end{theorem}

\begin{proof}

From $\triangle A_{2}A_{0}A_{3},$ we obtain:

\[\angle A_{1}A_{2}A_{3}-\epsilon+\angle A_{2}A_{1}A_{3}+2\epsilon+\angle A_{1}A_{3}A_{2}-k\epsilon=\pi,\]
which yields

\[k=1.\]

By applying the law of sines in $\triangle A_{0}A_{2}A_{1},$ $\triangle A_{0}A_{3}A_{1},$
we derive:

\begin{equation}\label{alpha102sin}
\frac{(a_{01})_{0}}{\sin \epsilon}=\frac{(a_{12})_{0}}{\sin \angle A_{1}A_{0}A_{2}}=
\frac{(a_{13})_{0}}{-\sin (\angle A_{2}A_{1}A_{3}+2\epsilon+\angle A_{1}A_{0}A_{2} )}.
\end{equation}

From (\ref{alpha102sin}), we obtain (\ref{alpha102}).

By replacing (\ref{alpha102}), (\ref{alpha203}), (\ref{alpha103}) in (\ref{inverse111}), we obtain
$B_{i}=B_{i}(\epsilon),$ for $i=1,2,3.$

\end{proof}

\begin{corollary}
For $\epsilon\to 0,$ we derive a degenerate weighted Fermat-Torricelli tree $\{A_{2}A_{1}, A_{1}A_{3}\}.$
\end{corollary}

\section{An $\epsilon$ characterization of the vertices of a geodesic triangle on the $K$-plane}

Let $\triangle A_{1}A_{2}A_{3}$ be a geodesic triangle on the $K$ plane.
The $K$ plane is a sphere $S_{K}^{2}$ of radius $R=\frac{1}{\sqrt{K}}$
and a hyperbolic plane $H_{K}^{2}$ with constant Gaussian curvature $-K$ for $K<0$
in $\mathbb{R}^{3}.$

We set
\begin{displaymath}
\kappa = \left\{ \begin{array}{ll}
\sqrt{K} & \textrm{if $K>0$,}\\
i\sqrt{-K} & \textrm{if $K<0$.}\\
\end{array} \right.
\end{displaymath}

The unified law of cosines and law of sines on the $K$ plane is given in \cite{BNik:07}.

We set $\angle A_{1}A_{2}A_{0}= \|\epsilon\|,$ $\angle A_{2}A_{0}A_{3}= \angle A_{2}A_{1}A_{3}+2\epsilon,$
$\angle A_{0}A_{3}A_{1}= \frac{\|\epsilon\|}{2}.$



We set

\begin{displaymath}
\epsilon = \left\{ \begin{array}{ll}
\|\epsilon \| & \textrm{if $K>0$,}\\
-\|\epsilon \| & \textrm{if $K<0$.}\\
\end{array} \right.
\end{displaymath}

\begin{theorem}\label{echarKplane}

The weight $B_{i}=B_{i}(\epsilon)$ are uniquely determined by the formula:
\begin{equation}\label{inverseB1eK}
B_{1}=\frac{C}{1+\frac{\sin{\angle A_{1}A_{0}A_{3}}}{\sin{\angle A_{2}A_{0}A_{3}}}+\frac{\sin{\angle A_{1}A_{0}A_{2}}}{\sin{\angle A_{2}A_{0}A_{3}}}},
\end{equation}

\begin{equation}\label{inverseB2eK}
B_{2}=\frac{C}{1+\frac{\sin{\angle A_{2}A_{0}A_{3}}}{\sin{\angle A_{1}A_{0}A_{3}}}+\frac{\sin{\angle A_{2}A_{0}A_{1}}}{\sin{\angle A_{1}A_{0}A_{3}}}},
\end{equation}

\begin{equation}\label{inverseB3eK}
B_{3}=\frac{C}{1+\frac{\sin{\angle A_{3}A_{0}A_{1}}}{\sin{\angle A_{1}A_{0}A_{2}}}+\frac{\sin{\angle A_{3}A_{0}A_{2}}}{\sin{\angle A_{1}A_{0}A_{2}}}},
\end{equation}

where

\begin{equation}\label{alpha203K}
\angle A_{2}A_{0}A_{3}= \angle A_{2}A_{1}A_{3}+2\epsilon
\end{equation}

\begin{equation}\label{alpha102K}
\angle A_{1}A_{0}A_{2}=\operatorname{arccot} (-\frac{\sin(\kappa(a_{13})_{g})+2\sin(\kappa(a_{12})_{g})\cos(\frac{\|\epsilon\|}{2})\cos(\angle A_{2}A_{1}A_{3}+2\epsilon)}{2\sin\kappa(a_{12})_{g})\cos(\frac{\|\epsilon\|}{2})\sin(\angle A_{2}A_{1}A_{3}+2\epsilon)}),
\end{equation}

\begin{equation}\label{alpha103K}
\angle A_{1}A_{0}A_{3}=2\pi- \angle A_{2}A_{1}A_{3}-2\epsilon-\angle A_{1}A_{0}A_{2}(\epsilon).
\end{equation}

\end{theorem}

\begin{proof}
By applying the law of sines in $\triangle A_{0}A_{2}A_{1},$ we get:

\begin{equation}\label{lawksine1}
\frac{\sin (\kappa (a_{10})_{g})}{\sin\|\epsilon\|}=\frac{\sin (\kappa (a_{12})_{g})}{\sin\angle A_{2}A_{0}A_{1}}
\end{equation}

or

\begin{equation}\label{lawksine12}
\frac{\sin (\kappa (a_{10})_{g})}{\sin\frac{|\epsilon\|}{2}}\frac{1}{2 \cos\frac{|\epsilon\|}{2}}=\frac{\sin (\kappa (a_{12})_{g})}{\sin\angle A_{2}A_{0}A_{1}}.
\end{equation}

By applying the law of sines in $\triangle A_{0}A_{1}A_{3},$ we get:

\begin{equation}\label{lawksine2}
\frac{\sin (\kappa (a_{10})_{g})}{\sin\frac{\|\epsilon\|}{2}}=-\frac{\sin (\kappa (a_{13})_{g})}{\sin(\angle A_{2}A_{1}A_{3}+2\epsilon+\angle A_{2}A_{0}A_{1})}
\end{equation}

By replacing (\ref{lawksine2}) in (\ref{lawksine12}), we derive:

\begin{equation}\label{lawksine12s}
-\frac{\sin (\kappa (a_{13})_{g})}{\sin(\angle A_{2}A_{1}A_{3}+2\epsilon+\angle A_{2}A_{0}A_{1})}\frac{1}{2 \cos\frac{|\epsilon\|}{2}}=\frac{\sin (\kappa (a_{12})_{g})}{\sin\angle A_{2}A_{0}A_{1}}.
\end{equation}

By solving (\ref{lawksine12s}) w.r. to $\angle A_{2}A_{0}A_{1},$ we obtain (\ref{alpha102K}).

\end{proof}

\section{An $\epsilon$ characterization of the vertices of a geodesic triangle on flat surfaces of revolution}

In this section, we shall give an $\epsilon$ characterization of the vertices of geodesic triangles
on flat surfaces of revolution (Circular Cylinder $S$ and  Circular Cone $C$) which are flat Euclidean Surfaces with zero Gaussian curvature.

\subsection{An $\epsilon$ characterization of the vertices of a geodesic triangle on a circular cylinder.}

The parametric form of a (right) circular cylinder $S$
of unit radius  and axis (axis of revolution) the z-axis:

$\vec{r}(u,v)=(\cos v,\sin v,u).$

The geodesics of the circular cylinder are
the straight lines on the circular cylinder parallel to the
$z$-axis, the circles obtained by intersecting the circular
cylinder with planes parallel to the $xy$-plane and circular
helixes of the parametric form $\vec{r}(t)=(\cos t, \sin t,
bt+c).$

Let $\triangle A_{1}A_{2}A_{3}$ be a geodesic triangle on $S$ which is
composed of three circular helixes.

We set
$A_{1}=(1,0,0),$ $A_{2}=(\cos\varphi_{2},\sin\varphi_{2},z_{2})$
$A_{3}=(\cos\varphi_{3},\sin\varphi_{3},z_{3})$ and
$\vec{r}_{ij}=(\cos t, \sin t, b_{ij} t)$ the
circular helix on $S$ from $A_{i}$ to $A_{j}$ for $i,j=1,2,3, i\ne
j$ and $\varphi_{2},\varphi_{3}\in (0,\pi).$

The coefficient $b_{ij}$ is called the step of the helix from $A_{i}$ to $A_{j}.$
The step of the helices $b_{12}$ from $A_{1}$ to $A_{2}$ and $b_{13}$ from $A_{1}$ to $A_{3}$
are given by:
\[b_{12}=\frac{z_{2}}{\varphi_{2}}\]
and
\[b_{13}=\frac{z_{3}}{\varphi_{3}}.\]

A cylindrical law of cosines for geodesic
triangles on $S$ composed of three circular helixes is given in \cite{Zach:14a}.
\begin{proposition}\cite{Zach:14a}\label{propcoslaw} The following formula holds
for $\triangle A_{1}A_{2}A_{3}$ on S:

\begin{equation}\label{equationcoslaw}
(1+b_{23}^{2})(\varphi_{2}-\varphi_{3})^{2}=(1+b_{12}^{2})\varphi_{2}^{2}+(1+b_{13}^{2})\varphi_{3}^{2}-2\sqrt{(1+b_{12}^{2})(1+b_{13}^{2})}\varphi_{2}\varphi_{3}\cos\alpha_{213}.
\end{equation}

\end{proposition}

We denote by $(a_{ij})_{S}$ the length of the geodesic arc from $A_{i}$ to $A_{j}.$

\begin{theorem}\label{echarR2}

The weight $B_{i}=B_{i}(\epsilon)$ are uniquely determined by the formula:
\begin{equation}\label{inverseB1ess}
B_{1}=\frac{C}{1+\frac{\sin{\angle A_{1}A_{0}A_{3}}}{\sin{\angle A_{2}A_{0}A_{3}}}+\frac{\sin{\angle A_{1}A_{0}A_{2}}}{\sin{\angle A_{2}A_{0}A_{3}}}},
\end{equation}

\begin{equation}\label{inverseB2ess}
B_{2}=\frac{C}{1+\frac{\sin{\angle A_{2}A_{0}A_{3}}}{\sin{\angle A_{1}A_{0}A_{3}}}+\frac{\sin{\angle A_{2}A_{0}A_{1}}}{\sin{\angle A_{1}A_{0}A_{3}}}},
\end{equation}

\begin{equation}\label{inverseB3ess}
B_{3}=\frac{C}{1+\frac{\sin{\angle A_{3}A_{0}A_{1}}}{\sin{\angle A_{1}A_{0}A_{2}}}+\frac{\sin{\angle A_{3}A_{0}A_{2}}}{\sin{\angle A_{1}A_{0}A_{2}}}},
\end{equation}

where

\begin{equation}\label{alpha203ss}
\angle A_{2}A_{0}A_{3}= \angle A_{2}A_{1}A_{3}+2\epsilon
\end{equation}

\begin{equation}\label{alpha102ss}
\angle A_{1}A_{0}A_{2}=\angle A_{1}A_{0}A_{2}(\epsilon)=\operatorname{arccot} (-\frac{\sqrt{1+b_{13}^2}\varphi_{3}+\sqrt{1+b_{12}^2}\varphi_{2}\cos(\angle A_{2}A_{1}A_{3}+2\epsilon)}{(\sqrt{1+b_{12}^2}\varphi_{2}\sin(\angle A_{2}A_{1}A_{3}+2\epsilon)}),
\end{equation}

\begin{equation}\label{alpha103ss}
\angle A_{1}A_{0}A_{3}=2\pi- \angle A_{2}A_{1}A_{3}-2\epsilon-\angle A_{1}A_{0}A_{2}(\epsilon).
\end{equation}

\end{theorem}

\begin{proof}

Unrolling the cylinder $S$ in terms of the vertex $A_{1},$ we derive
an isometric mapping from $S$ to $\mathbb{R}^{2},$ which yields:

\begin{equation}\label{isomlengths}
(a_{ij})_{S}=(a_{ij})_{0}, for\quad i,j=1,2,3.
\end{equation}

From (\ref{isomlengths}), we have:

\begin{equation}\label{a12ss}
(a_{12})_{S}=(a_{12})_{0}
\end{equation}

and

\begin{equation}\label{a13ss}
(a_{13})_{S}=(a_{13})_{0}
\end{equation}

The Euclidean distances $(a_{12})_{0}$ and $(a_{13})_{0}$ are given by:

\begin{equation}\label{a1200}
(a_{12})_{0}=\sqrt{z_{2}^2+\varphi_{2}^2}
\end{equation}

or
\begin{equation}\label{a1200s}
(a_{12})_{0}=\sqrt{b_{12}^2+1} \varphi_{2}
\end{equation}

and

\begin{equation}\label{a1300}
(a_{13})_{0}=\sqrt{z_{3}^2+\varphi_{3}^2}
\end{equation}

or
\begin{equation}\label{a1300s}
(a_{13})_{0}=\sqrt{b_{13}^2+1} \varphi_{3}
\end{equation}

Thus, the length of the circular helix with parametric form $\vec{r}_{12}$ from $A_{1}$
to $A_{2}$ and with parametric form $\vec{r}_{13}$ from $A_{1}$ to $A_{3}$ is given by:

\begin{equation}\label{a12ss}
(a_{12})_{s}=\sqrt{1+b_{12}^2}\varphi_{2}.
\end{equation}

and

\begin{equation}\label{a13ss}
(a_{13})_{s}=\sqrt{1+b_{13}^2}\varphi_{3}.
\end{equation}

By replacing (\ref{a12ss}) and (\ref{a13ss}) in \ref{alpha102}, we obtain (\ref{alpha102ss})
and by applying Theorem~1, we derive (\ref{inverseB1ess}), (\ref{inverseB2ess}) and (\ref{inverseB3ess}).
The weights $B_{1},$ $B_{2},$ $B_{3}$ depend on the angle $\angle A_{2}A_{1}A_{3},$ $\epsilon$ and the step of the helices $b_{12}$ and $b_{13}.$

\end{proof}


\subsection{An $\epsilon$ characterization of the vertices of a geodesic triangle on a circular cone.}

We consider the parametric form of a (right) circular cone $C$ with a
unit base radius.

$\vec{r}(u,v)=\left(\left(1-\frac{u}{H}\right)\cos
v,\left(1-\frac{u}{H}\right)\sin v,u\right),$ $ 0<u \le H,
0<v<2\pi.$

The geodesic equations on  $C$ are given in
\cite[Exercise~5.2.14, Subsection~5.6.2 pp.~222,
pp.~247-248]{Oprea:07}.

Let $\triangle A_{1}A_{2}A_{3}$ be a geodesic triangle on
$C.$

We denote by $P$ the center of the unit bases circle of
$S^{\prime},$ $H$ the distance $AP,$ by $A_{ip}$ the intersection of the line defined by
the linear segment $AA_{i}$ with the unit bases circle $c(P,1)$
for $i=0,1,2,3,$  by $\varphi_{2}$ the angle $\angle A_{1}PA_{2p}$
by $\varphi_{3}$ the angle $\angle A_{1}PA_{3p}$ and by
$\varphi_{0}$ the angle $\angle A_{1}PA_{0p}.$ 

By unrolling the circular cone $C$ w.r. to $A_{1}A,$ (cut along $A_{1}A$) we
derive an isometric mapping from $C$ to $\mathbb{R}^{2}.$

Thus, we get:
\begin{equation}\label{isomcone1}
(a_{ij})_{g}=(a_{ij})_{0}
\end{equation}

By setting $A_{1}=(0,0),$ we obtain:

\begin{equation}\label{anglephi}
\varphi=\frac{2\pi}{\sqrt{1+H^{2}}},
\end{equation}

\begin{equation}\label{anglephi2}
\angle A_{1}AA_{2}=\frac{\varphi_{2}}{\sqrt{1+H^{2}}},
\end{equation}

\begin{equation}\label{anglephi3}
\angle A_{1}AA_{3}=\frac{\varphi_{3}}{\sqrt{1+H^{2}}}
\end{equation}

and

\begin{equation}\label{anglephi0}
\angle A_{1}AA_{0}=\frac{\varphi_{0}}{\sqrt{1+H^{2}}}.
\end{equation}

where

\begin{equation}\label{aia}
A_{i}A=\sqrt{1+H^2}.
\end{equation}

\begin{theorem}\label{echarR2c}

The weight $B_{i}=B_{i}(\epsilon)$ are uniquely determined by the formula:
\begin{equation}\label{inverseB1ec}
B_{1}=\frac{C}{1+\frac{\sin{\angle A_{1}A_{0}A_{3}}}{\sin{\angle A_{2}A_{0}A_{3}}}+\frac{\sin{\angle A_{1}A_{0}A_{2}}}{\sin{\angle A_{2}A_{0}A_{3}}}},
\end{equation}

\begin{equation}\label{inverseB2ec}
B_{2}=\frac{C}{1+\frac{\sin{\angle A_{2}A_{0}A_{3}}}{\sin{\angle A_{1}A_{0}A_{3}}}+\frac{\sin{\angle A_{2}A_{0}A_{1}}}{\sin{\angle A_{1}A_{0}A_{3}}}},
\end{equation}

\begin{equation}\label{inverseB3ec}
B_{3}=\frac{C}{1+\frac{\sin{\angle A_{3}A_{0}A_{1}}}{\sin{\angle A_{1}A_{0}A_{2}}}+\frac{\sin{\angle A_{3}A_{0}A_{2}}}{\sin{\angle A_{1}A_{0}A_{2}}}},
\end{equation}

where

\begin{equation}\label{alpha203c}
\angle A_{2}A_{0}A_{3}= \angle A_{2}A_{1}A_{3}+2\epsilon
\end{equation}

\begin{equation}\label{alpha102c}
\angle A_{1}A_{0}A_{2}=\angle A_{1}A_{0}A_{2}(\epsilon)=\operatorname{arccot} (-\frac{(a_{13})_{c}+(a_{12})_{c}\cos(\angle A_{2}A_{1}A_{3}+2\epsilon)}{(a_{12})_{c}\sin(\angle A_{2}A_{1}A_{3}+2\epsilon)}),
\end{equation}

\begin{equation}\label{alpha103c}
\angle A_{1}A_{0}A_{3}=2\pi- \angle A_{2}A_{1}A_{3}-2\epsilon-\angle A_{1}A_{0}A_{2}(\epsilon).
\end{equation}

and

\begin{equation}\label{a12c}
(a_{12})_{c}=\sqrt{(1+H^2)+(A_{2}A)^2-2\sqrt{1+H^2}(A_{2}A)\cos(\frac{\varphi_{2}}{\sqrt{1+H^{2}}})} ,
\end{equation}

\begin{equation}\label{a13c}
(a_{13})_{c}=\sqrt{(1+H^2)+(A_{3}A)^2-2\sqrt{1+H^2}(A_{3}A)\cos(\frac{\varphi_{3}}{\sqrt{1+H^{2}}})}.
\end{equation}

\end{theorem}

\begin{proof}

From the law of cosines in $\triangle A_{1}AA_{2},$ we get: 
\begin{equation}\label{a12c2}
(a_{12})_{c}=\sqrt{(A_{1}A)+(A_{2}A)^2-2(A_{1}A)(A_{2}A)\cos(\angle A_{1}AA_{2})} .
\end{equation}

By replacing (\ref{aia}) and (\ref{anglephi2}) in (\ref{a12c2}),
we obtain (\ref{a12c}).

From the law of cosines in $\triangle A_{1}AA_{3},$ we get:

\begin{equation}\label{a13c2}
(a_{13})_{c}=\sqrt{(A_{1}A)^2+(A_{3}A)^2-2 A_{1}A(A_{3}A)\cos(\angle A_{1}AA_{3})}.
\end{equation}

By replacing (\ref{aia}) and (\ref{anglephi3}) in (\ref{a13c2}),
we obtain (\ref{a13c}).
Unrolling $C$ along $A_{1}A$ yields an isometric mapping from $C$ to $\mathbb{R}^{2}$ and by applying theorem~1,
we obtain (\ref{inverseB1ec}), (\ref{inverseB2ec}) and (\ref{inverseB3ec}).
The weights $B_{1},$ $B_{2},$ $B_{3}$ depend on $\epsilon,$ $\varphi_{2},$ $\varphi_{3}$ and $H.$ 

\end{proof}

By setting $A_{1}\equiv A,$ we obtain an $\epsilon$ characterization of the conical vertex $A$ in $\mathbb{R}^3.$

\begin{proposition}\label{echarR3c}

The weight $B_{i}=B_{i}(\epsilon)$ are uniquely determined by the formula:
\begin{equation}\label{inverseB1ec}
B_{1}=\frac{C}{1+\frac{\sin{\angle A_{1}A_{0}A_{3}}}{\sin{\angle A_{2}A_{0}A_{3}}}+\frac{\sin{\angle A_{1}A_{0}A_{2}}}{\sin{\angle A_{2}A_{0}A_{3}}}},
\end{equation}

\begin{equation}\label{inverseB2ec}
B_{2}=\frac{C}{1+\frac{\sin{\angle A_{2}A_{0}A_{3}}}{\sin{\angle A_{1}A_{0}A_{3}}}+\frac{\sin{\angle A_{2}A_{0}A_{1}}}{\sin{\angle A_{1}A_{0}A_{3}}}},
\end{equation}

\begin{equation}\label{inverseB3ec}
B_{3}=\frac{C}{1+\frac{\sin{\angle A_{3}A_{0}A_{1}}}{\sin{\angle A_{1}A_{0}A_{2}}}+\frac{\sin{\angle A_{3}A_{0}A_{2}}}{\sin{\angle A_{1}A_{0}A_{2}}}},
\end{equation}

where

\begin{equation}\label{alpha203c}
\angle A_{2}A_{0}A_{3}= \angle A_{2}A_{1}A_{3}+2\epsilon
\end{equation}

\begin{equation}\label{alpha102c}
\angle A_{1}A_{0}A_{2}=\angle A_{1}A_{0}A_{2}(\epsilon)=\operatorname{arccot} (-\frac{AA_{3}+AA_{2}\cos(\angle A_{2}A_{1}A_{3}+2\epsilon)}{(AA_{2}\sin(\angle A_{2}A_{1}A_{3}+2\epsilon)}),
\end{equation}

\begin{equation}\label{alpha103c}
\angle A_{1}A_{0}A_{3}=2\pi- \angle A_{2}A_{1}A_{3}-2\epsilon-\angle A_{1}A_{0}A_{2}(\epsilon).
\end{equation}

\end{proposition}

\begin{remark}
The deviation of $\|B_{1}-B_{1}(\epsilon) \|$
where

\[B_{1}=\sqrt{B_{2}(\epsilon)^2+B_{3}(\epsilon)^2+2 B_{2}(\epsilon)B_{3}(\epsilon)\cos\angle A_{3}A_{1}A_{2}}\]
gives an error estimate which depends on $\epsilon.$
\end{remark}


\begin{thebibliography}{99}
\bibitem{BNik:07}
I.D. Berg and I.G. Nikolaev, \emph{ On an extremal property of
quadrilaterals in an Aleksandrov space of curvature $\leq K$. The
interaction of analysis and geometry}, Contemp. Math. \textbf{424}
(2007)  1--15.

\bibitem{Cots/Zach:11} A. Cotsiolis and A. Zachos, \emph{Corrigendum to "The weighted Fermat-Torricelli problem on a surface and an "inverse"
problem"}, \emph{J. Math. Anal. Appl.}, \textbf{376}, no. 2 (2011)
760.
\bibitem{Oprea:07} J. Oprea, \emph{Differential Geometry and its Applications}
Washington, Mathematical Association of America, 2007.
\bibitem{Zachos/Cots:10} A. Zachos and A. Cotsiolis, \emph{The weighted Fermat-Torricelli problem on a surface and an "inverse" problem},  \emph{J. Math. Anal. Appl.}, \textbf{373}, no. 1 (2011)  44--58.
\bibitem{Zachos:13} A. Zachos, \emph{Location of the weighted Fermat-Torricelli point on the
K-plane.} Analysis, München \textbf{33}, no. 3 (2013), 243-249.

\bibitem{Zachos 14b}A. Zachos, \emph{Location of the weighted Fermat-Torricelli point on the
K-plane. II} Analysis, München \textbf{34}, no. 1 (2014), 111-120.

\bibitem{Zach:14a} A. Zachos, \emph{Exact Location of the Weighted Fermat-Torricelli point on Flat Surfaces of Revolution}, Results. Math \textbf{65} (2014), 167-179.
\bibitem{Zach:15} A. Zachos, \emph{An analytical solution of the weighted Fermat-Torricelli problem on a unit sphere.} Rend. Circ. Mat. Palermo (2) \textbf{64}, no. 3 (2015), 451-458.


\end{thebibliography}
\end{document}